\pdfoutput=1
\documentclass[a4paper, 11pt, reqno]{amsart}

\usepackage{amssymb}
\usepackage{amsmath}
\usepackage{amsthm}
\usepackage[numbers]{natbib}
\usepackage{mathtools}
\usepackage{enumitem}
\usepackage{calc}
\usepackage{setspace}
\usepackage{graphicx}
\usepackage{tikz-cd}
\usepackage{mathscinet}
\usepackage{hyperref}
\newcommand{\noop}[1]{}
\usepackage[top=3 cm, bottom=3 cm, left=3 cm, right=3cm]{geometry}

\tikzset{
  symbol/.style={
    draw=none,
    every to/.append style={
      edge node={node [sloped, allow upside down, auto=false]{$#1$}}}
  }
}

\newcounter{dummy}
\newtheorem{thm}[dummy]{Theorem}

\newtheorem{prop}[dummy]{Proposition}
\newtheorem{cor}[dummy]{Corollary}
\newtheorem{qstn}[dummy]{Question}
\theoremstyle{definition}
\newtheorem{rmk}[dummy]{Remark}

\DeclareMathOperator{\aff}{\mathrm{aff}}
\DeclareMathOperator{\ant}{\mathrm{ant}}

\DeclareMathOperator{\id}{id}

\DeclareMathOperator{\Gal}{\mathrm{Gal}}

\DeclareMathOperator{\Lie}{\mathrm{Lie}}

\DeclareMathOperator{\Hom}{\mathrm{Hom}}
\DeclareMathOperator{\Mor}{\mathrm{Mor}}

\makeatletter
\renewcommand{\@biblabel}[1]{[#1]\hfill}
\makeatother

\begin{document}
\title[On morphisms between conn. comm. alg. groups over a field of char. $0$]{On morphisms between connected commutative algebraic groups over a field of characteristic $0$}

\author[G. A. Dill]{Gabriel A. Dill}
\address{Leibniz Universit\"at Hannover, Institut f\"ur Algebra, Zahlentheorie und Diskrete Mathematik, Welfengarten 1, 30167 Hannover, Germany}
\email{dill@math.uni-hannover.de}

\date{\today}
\maketitle

\begin{abstract}
Let $K$ be a field of characteristic $0$ and let $G$ and $H$ be connected commutative algebraic groups over $K$. Let $\Mor_0(G,H)$ denote the set of morphisms of algebraic varieties $G \to H$ that map the neutral element to the neutral element. We construct a natural retraction from $\Mor_0(G,H)$ to $\Hom(G,H)$ (for arbitrary $G$ and $H$) which commutes with the composition and addition of morphisms. In particular, if $G$ and $H$ are isomorphic as algebraic varieties, then they are isomorphic as algebraic groups. If $G$ has no non-trivial unipotent group as a direct factor, we give an explicit description of the sets of all morphisms and isomorphisms of algebraic varieties between $G$ and $H$. We also characterize all connected commutative algebraic groups over $K$ whose only variety automorphisms are compositions of automorphisms of algebraic groups with translations.
\end{abstract}

In this note, we work over a ground field $K$ of characteristic $0$ that is not necessarily algebraically closed. We define two categories $\mathfrak{C}$ and $\mathfrak{D}$: they both have as objects the connected commutative algebraic groups over $K$. Given two such algebraic groups $G$ and $H$, the morphisms from $G$ to $H$ in $\mathfrak{C}$ are all homomorphisms of algebraic groups between them while the morphisms from $G$ to $H$ in $\mathfrak{D}$ are all morphisms of pointed algebraic varieties between them (with the neutral element being the respective distinguished point). There is a canonical inclusion functor $\mathcal{F}: \mathfrak{C} \to \mathfrak{D}$ as $\mathfrak{C}$ is a subcategory of $\mathfrak{D}$.

The existence statement in the following theorem and the construction in its proof were suggested by B. Kahn as a strengthening of an independent earlier proof of Corollary \ref{cor:main} below by the author. The proof of the automatic additivity and uniqueness of the constructed functor is due to an anonymous referee of this note.

\begin{thm}\label{thm:main}
There exists one and only one functor $\mathcal{G}: \mathfrak{D} \to \mathfrak{C}$ such that $\mathcal{G} \circ \mathcal{F} = \id_{\mathfrak{C}}$. Furthermore, this $\mathcal{G}$ commutes with the addition of morphisms, i.e., $\mathcal{G}(\phi_1+\phi_2) = \mathcal{G}(\phi_1) + \mathcal{G}(\phi_2)$ for all morphisms of pointed algebraic varieties $\phi_i: G \to H$ ($i = 1,2$).
\end{thm}

If $G$ has no non-trivial unipotent group as a direct factor, we will give an explicit description of the sets of all (not necessarily pointed) morphisms and isomorphisms of algebraic varieties between $G$ and $H$ (see Theorem \ref{thm:explicit} below). We will then use this description to classify all connected commutative algebraic groups $G$ over $K$ such that every automorphism of the pointed algebraic variety $G$, the neutral element being the distinguished point, is an automorphism of algebraic groups (see Theorem \ref{thm:characterization} below).

The idea of the proof of Theorem \ref{thm:main} is the following: to a connected commutative algebraic group over $K$, we can associate its antiaffine part, its toric part, and its unipotent part. A morphism of pointed algebraic varieties between two connected commutative algebraic groups then naturally induces homomorphisms between the respective antiaffine and toric parts as well as a morphism between the unipotent parts. Using the canonical isomorphism between a commutative unipotent algebraic group and the algebraic group associated to its Lie algebra in characteristic $0$, we replace this last morphism by its differential at the neutral element. We then check that these homomorphisms patch together to give a homomorphism on the whole domain and that the construction commutes with the composition of morphisms so that we obtain a functor $\mathcal{G}$ with the desired property. Finally, we show that $\mathcal{G}$ is unique with this property and that it automatically commutes with the addition of morphisms.

The notation introduced in the following well-known proposition coincides with the one used in \cite{Brion09} and will be used throughout this article:

\begin{prop}[Chapter 8 of \cite{MilneAG}, in particular Proposition 8.1, Example 8.4, Theorem 8.27, and Propositions 8.36 and 8.37]\label{prop:summary}
Let $G$ be a connected algebraic group over a field $K$ of characteristic $0$. Then $G$ contains a smallest normal algebraic $K$-subgroup $H$ such that $G/H$ is an abelian variety. This $H$ is connected and affine and we denote it by $G_{\aff}$. Furthermore, $G$ contains a smallest normal algebraic $K$-subgroup $H'$ such that $G/H'$ is affine. This $H'$ is connected and antiaffine, i.e., $\mathcal{O}(H') = K$, and we denote it by $G_{\ant}$.
\end{prop}

We now prove Theorem \ref{thm:main}.

\begin{proof}[Proof of Theorem \ref{thm:main}]
Let $G$ and $H$ be connected commutative algebraic groups over $K$ and let $0_G$ and $0_H$ denote the respective neutral elements of $G$ and $H$. Let $\phi: G \to H$ be a morphism of algebraic varieties such that $\phi(0_G) = 0_H$.

In order to prove the theorem, we have to construct a homomorphism of algebraic groups $\mathcal{G}(\phi): G \to H$. Furthermore, this construction should commute with composition and if $\phi$ is already a homomorphism of algebraic groups, we should have $\mathcal{G}(\phi) = \phi$. Lastly, we want to show that this $\mathcal{G}$ is unique and commutes with addition.

By Proposition 3.1(i) in \cite{Brion09} (a scheme-theoretic variant of a structure theorem due to Rosenlicht in \cite{Rosenlicht}), the addition homomorphism induces an isomorphism $(G_{\aff} \times_K G_{\ant})/(G_{\aff} \cap G_{\ant}) \simeq G$ and the same for $H$. By Lemma 1.5(i) in \cite{Brion09}, $\phi$ induces a homomorphism of algebraic groups $G_{\ant} \to H_{\ant}$.

Consider now the morphism $G_{\aff} \to H/H_{\aff}$ induced by $\phi$. By Corollary 3.6 in \cite{Milne}, this morphism is a homomorphism of algebraic groups since $H/H_{\aff}$ is an abelian variety and $G_{\aff}$ is connected. By Proposition 8.1 in \cite{MilneAG}, its image is a connected affine algebraic group. But since $H/H_{\aff}$ is an abelian variety, this image must be trivial by Example 8.4 in \cite{MilneAG}, so $\phi$ induces a morphism of algebraic varieties $G_{\aff} \to H_{\aff}$.

By Theorem 16.13 in \cite{MilneAG}, we have a decomposition $G_{\aff} \simeq (G_{\aff})_s \times_K (G_{\aff})_u$ into a part of multiplicative type $(G_{\aff})_s \subset G_{\aff}$ and a unipotent part $(G_{\aff})_u \subset G_{\aff}$ and similarly for $H_{\aff}$. Since $G_{\aff}$ and $H_{\aff}$ are connected, so are $(G_{\aff})_s$ and $(H_{\aff})_s$. We identify $G_{\aff}$ with $(G_{\aff})_s \times_K (G_{\aff})_u$.

By Proposition 12.49 in \cite{MilneAG}, the composition $G_{\aff} \to H_{\aff} \to (H_{\aff})_s$ of the restriction of $\phi$ with the canonical projection is a homomorphism of algebraic groups. This homomorphism then has to factor through $(G_{\aff})_s$ by Corollary 14.18(a) in \cite{MilneAG}. So $\phi$ induces a homomorphism of algebraic groups $(G_{\aff})_s \to (H_{\aff})_s$ as well as a morphism of algebraic varieties $(G_{\aff})_u \to (H_{\aff})_u$.

We now construct a homomorphism $\psi: G_{\aff} \times_K G_{\ant} = (G_{\aff})_s \times_K (G_{\aff})_u \times_K G_{\ant} \to H$. Equivalently, we construct three homomorphisms from $(G_{\aff})_s$, $(G_{\aff})_u$, and $G_{\ant}$ to $H$ respectively: On $G_{\ant}$, we take the restriction of $\phi$. On $(G_{\aff})_s$, we take the induced homomorphism of algebraic groups to $(H_{\aff})_s$ from above. On $(G_{\aff})_u$, we take the composition $(G_{\aff})_u \to \Lie (G_{\aff})_u \to \Lie (H_{\aff})_u \to (H_{\aff})_u$, where the morphism in the middle is the differential at $0_G$ of the restriction of $\phi$ and the two other morphisms are the canonical isomorphisms between a commutative unipotent algebraic group over $K$ and the algebraic group associated to its Lie algebra (see Proposition 14.32 in \cite{MilneAG}).

We will now show that the following claim holds: if $G' \subset G_{\aff}$ is any algebraic subgroup such that $\phi|_{G'}: G' \to H$ is a homomorphism, then $\psi|_{G' \times_K \{0_G\}} = \phi|_{G'}$ (if we identify $G'$ and $G' \times_K \{0_G\}$). To prove this, let $G'_s$ and $G'_u$ denote the part of multiplicative type and the unipotent part of $G'$ respectively, then it suffices to show that $\psi|_{G'_s \times_K \{0_G\}} = \phi|_{G'_s}$ and $\psi|_{G'_u \times_K \{0_G\}} = \phi|_{G'_u}$. 

By Corollary 14.18(b) in \cite{MilneAG}, we have $\phi(G'_s) \subset (H_{\aff})_s$ and so $\psi|_{G'_s \times_K \{0_G\}} = \phi|_{G'_s}$ by the construction of $\psi$. Thanks to the commutative diagram before Proposition 14.32 in \cite{MilneAG}, the homomorphism $\psi|_{G'_u \times_K \{0_G\}}$ is by construction equal to the composition $G'_u \to \Lie G'_u \to \Lie (H_{\aff})_u \to (H_{\aff})_u$, where the morphism in the middle is the differential at $0_G$ of $\phi|_{G'_u}$, the two other morphisms are the canonical isomorphisms between a commutative unipotent algebraic group over $K$ and the algebraic group associated to its Lie algebra, and we identify $G'_u$ with $G'_u \times_K \{0_G\}$. It now follows from the commutative diagram before Proposition 14.32 in \cite{MilneAG} that $\psi|_{G'_u \times_K \{0_G\}} = \phi|_{G'_u}$. Thus, the above claim holds.

In particular, since $\phi|_{G'}$ is a homomorphism for $G' = G_{\aff} \cap G_{\ant}$, the claim implies that $\psi: G_{\aff} \times_K G_{\ant} \to H$ induces a homomorphism $\chi: G \to H$. We even get that $\chi = \phi$ if $\phi$ is already a homomorphism of algebraic groups. It is also clear that the construction commutes with composition, so we set $\mathcal{G}(\phi) = \chi$.

It remains to prove that there exists only one functor $\mathcal{G}: \mathfrak{D} \to \mathfrak{C}$ such that $\mathcal{G} \circ \mathcal{F} = \id_{\mathfrak{C}}$ and that this $\mathcal{G}$ commutes with addition. Let therefore $\mathcal{G}$ denote an arbitrary functor from $\mathfrak{D}$ to $\mathfrak{C}$ such that $\mathcal{G} \circ \mathcal{F} = \id_{\mathfrak{C}}$. The objects of $\mathfrak{C}$ and $\mathfrak{D}$ are the same and $\mathcal{F}$ fixes each object, so $\mathcal{G}$ must also fix each object.

Furthermore, if $\phi_i: G \to H$ ($i = 1,2$) are two morphisms of $\mathfrak{D}$, then it follows that $(\phi_1,\phi_2): G \to H \times_K H$ is a morphism of $\mathfrak{D}$. If $\alpha: H \times_K H \to H$ denotes the addition morphism, we deduce that $\mathcal{G}(\phi_1+\phi_2) = \mathcal{G}(\alpha) \circ \mathcal{G}((\phi_1,\phi_2)) = \alpha \circ \mathcal{G}((\phi_1,\phi_2))$ since $\alpha$ is a homomorphism of algebraic groups and $\mathcal{G} \circ \mathcal{F} = \id_{\mathfrak{C}}$. Let $p_i: H \times_K H \to H$ ($i=1,2$) denote the two canonical projections, which are also homomorphisms of algebraic groups. We compute that $p_i \circ \mathcal{G}((\phi_1,\phi_2)) = \mathcal{G}(p_i) \circ \mathcal{G}((\phi_1,\phi_2)) = \mathcal{G}(p_i \circ (\phi_1,\phi_2)) = \mathcal{G}(\phi_i)$ ($i=1,2$). We deduce that $\mathcal{G}((\phi_1,\phi_2)) = (\mathcal{G}(\phi_1),\mathcal{G}(\phi_2))$. Putting everything together, we deduce that $\mathcal{G}(\phi_1 + \phi_2) = \mathcal{G}(\phi_1) + \mathcal{G}(\phi_2)$, so $\mathcal{G}$ commutes with addition.

Let now $\phi: G \to H$ denote an arbitrary morphism of $\mathfrak{D}$ and let $(G_{\aff})_s$, $(G_{\aff})_u$, etc. be as above. The addition homomorphism of algebraic groups $\sigma: G_{\ant} \times_K (G_{\aff})_s \times_K (G_{\aff})_u \to G$ is surjective. Since $\mathcal{G}(\phi \circ \sigma) = \mathcal{G}(\phi) \circ \sigma$, $\mathcal{G}(\phi)$ is uniquely determined by $\mathcal{G}(\phi \circ \sigma)$ thanks to Theorem 5.13 in \cite{MilneAG}. For $H \in \{G_{\ant}, (G_{\aff})_s, (G_{\aff})_u\}$, let $i_H: H \to G_{\ant} \times_K (G_{\aff})_s \times_K (G_{\aff})_u$ and $\mathrm{pr}_H: G_{\ant} \times_K (G_{\aff})_s \times_K (G_{\aff})_u \to H$ denote the canonical inclusion and projection homomorphisms respectively. Since $\mathcal{G}(\phi \circ \sigma)$ is a homomorphism of algebraic groups, we have
\begin{align*}
\mathcal{G}(\phi \circ \sigma) =  \sum_{H \in \{G_{\ant}, (G_{\aff})_s, (G_{\aff})_u\}} \mathcal{G}(\phi \circ \sigma) \circ i_H \circ \mathrm{pr}_H\\
= \sum_{H \in \{G_{\ant}, (G_{\aff})_s, (G_{\aff})_u\}}{\mathcal{G}(\phi|_H) \circ \mathrm{pr}_H}.
\end{align*}
Thus, it suffices to show that $\mathcal{G}\left(\phi|_{G_{\ant}}\right)$, $\mathcal{G}\left(\phi|_{(G_{\aff})_s}\right)$, and $\mathcal{G}\left(\phi|_{(G_{\aff})_u}\right)$ are uniquely determined. We have seen above that $\phi|_{G_{\ant}}$ is a homomorphism of algebraic groups and so $\mathcal{G}\left(\phi|_{G_{\ant}}\right) = \phi|_{G_{\ant}}$.

We have also seen above that $\phi\left((G_{\aff})_s\right) \subset H_{\aff}$. Let $\tilde{\phi}: (G_{\aff})_s \to H_{\aff}$ denote the induced morphism. Since the inclusion morphism $H_{\aff} \hookrightarrow H$ is a homomorphism of algebraic groups, it follows that $\mathcal{G}\left(\phi|_{(G_{\aff})_s}\right)$ is uniquely determined by $\mathcal{G}(\tilde{\phi})$. We then deduce from the definition of $(H_{\aff})_s$ in Theorem 16.13 in \cite{MilneAG} that the image of $\mathcal{G}(\tilde{\phi})$ is contained in $(H_{\aff})_s$. Let $\iota: (H_{\aff})_s \to H_{\aff}$ and $\pi: H_{\aff} \to (H_{\aff})_s$ denote the canonical inclusion and projection homomorphisms respectively. By Proposition 12.49 in \cite{MilneAG}, $\pi \circ \tilde{\phi}$ is a homomorphism of algebraic groups. We deduce that
\[ \mathcal{G}(\tilde{\phi}) = \iota \circ \pi \circ \mathcal{G}(\tilde{\phi}) = \iota \circ \mathcal{G}(\pi \circ \tilde{\phi}) = \iota \circ \pi \circ \tilde{\phi}\]
is uniquely determined. By the above, $\mathcal{G}\left(\phi|_{(G_{\aff})_s}\right)$ is uniquely determined as well.

It remains to show that $\mathcal{G}\left(\phi|_{(G_{\aff})_u}\right)$ is uniquely determined. We have seen above that $\phi\left((G_{\aff})_u\right) \subset (H_{\aff})_u$. Since the inclusion morphism $(H_{\aff})_u \hookrightarrow H$ is a homomorphism of algebraic groups, $\mathcal{G}\left(\phi|_{(G_{\aff})_u}\right)$ does not depend on whether we regard $\phi|_{(G_{\aff})_u}$ as a morphism with target $H_{\aff}$ or as a morphism with target $(H_{\aff})_u \subset H_{\aff}$. We may therefore assume that $G = (G_{\aff})_u$ and $H = (H_{\aff})_u$ and have to show that $\mathcal{G}(\phi)$ is uniquely determined. By Proposition 14.32 in \cite{MilneAG}, both $G$ and $H$ are isomorphic to some power of the additive group $\mathbb{G}_{a,K}$. Using the fact that $\mathcal{G}$ commutes with composition and fixes homomorphisms of algebraic groups, pre-composing with the inclusion from some $\mathbb{G}_{a,K}$, and post-composing with the projection on another $\mathbb{G}_{a,K}$, we may therefore assume that $G = H = \mathbb{G}_{a,K} = \mathrm{Spec} K[X]$. Using that $\mathcal{G}$ commutes with addition, we may furthermore assume that $\phi$ is induced by the ring homomorphism $K[X] \to K[X]$ that maps $X$ to $c X^n$ for some $c \in K$ and some $n \in \mathbb{N} = \{1,2,\hdots\}$. Since the morphism induced by $X \mapsto c X$ is an endomorphism of algebraic groups and $\mathcal{G}$ commutes with composition and fixes homomorphisms of algebraic groups, we may even assume that $c = 1$.

We may also assume that $n \geq 2$ since $\mathcal{G}(\id_{\mathbb{G}_{a,K}}) = \id_{\mathbb{G}_{a,K}}$. We then want to show that $\mathcal{G}(\phi)$ is trivial. For $\lambda \in K$, let $m_{\lambda}: \mathbb{G}_{a,K} \to \mathbb{G}_{a,K}$ denote the multiplication-by-$\lambda$ homomorphism induced by the ring homomorphism $K[X] \to K[X]$ that maps $X$ to $\lambda X$. We have $\phi \circ m_2 = m_{2^n} \circ \phi$ and so
\[ \mathcal{G}(\phi) \circ m_2 = m_{2^n} \circ \mathcal{G}(\phi).\]
But $\mathcal{G}(\phi)$ is a homomorphism of algebraic groups, so $\mathcal{G}(\phi) = m_{\lambda}$ for some $\lambda \in K$. It follows that $2\lambda = 2^n\lambda$. Since $n \geq 2$, we deduce that $\lambda = 0$ and so $\mathcal{G}(\phi)$ is trivial as desired. This completes the proof of the theorem.
\end{proof}

We get the following corollary:

\begin{cor}\label{cor:main}
Let $K$ be a field of characteristic $0$ and let $G$ and $H$ be two connected commutative algebraic groups over $K$. If $G$ and $H$ are isomorphic as algebraic varieties over $K$, then they are isomorphic as algebraic groups.
\end{cor}

The question to what extent the variety structure does or does not determine the group structure of an algebraic group has also been studied recently by Popov in \cite{Popov}. Corollary \ref{cor:main} says that the group structure is completely determined by the variety structure if the algebraic group is connected, the ground field is of characteristic $0$, and we additionally impose \emph{a priori} that the algebraic group is commutative.

The commutativity hypothesis is crucial here: for example, every unipotent algebraic group over a field $K$ of characteristic $0$ is isomorphic as a variety to $\mathbb{A}^n_{K}$ for some non-negative integer $n$ by Proposition 14.32 in \cite{MilneAG}, but there exist non-commutative unipotent algebraic groups over $K$, which are therefore not isomorphic to $\mathbb{G}^n_{a,K}$ as algebraic groups for every non-negative integer $n$.

\begin{proof}[Proof of Corollary \ref{cor:main}]
Let $\phi: G \to H$ be an isomorphism of algebraic varieties and let $0_G$ and $0_H$ denote the respective neutral elements of $G$ and $H$. After composing with a translation, we can assume that $\phi(0_G) = 0_H$. Theorem \ref{thm:main} then furnishes an isomorphism of algebraic groups $\mathcal{G}(\phi): G \to H$.
\end{proof}

\begin{rmk}
If $K = \mathbb{C}$ and we weaken the hypotheses and only demand that the analytifications of $G$ and $H$ are isomorphic as complex-analytic spaces, then Corollary \ref{cor:main} is false, as pointed out to us by M. Hindry. For a concrete example, pointed out to us by Z. Gao, consider the variety $U$ over $\mathbb{C}$ from Section 7 of \cite{Neeman}; this variety $U$ is a non-trivial principal $\mathbb{A}^1_{\mathbb{C}}$-bundle over an elliptic curve $E = C$ over $\mathbb{C}$. By Theorem 7 on p. 185 of \cite{SerreBook}, the variety $U$ admits the structure of a connected commutative algebraic group over $\mathbb{C}$ that is an extension of $E$ by $\mathbb{G}_{a,\mathbb{C}}$. In Section 7 of \cite{Neeman}, it is shown that the analytifications of $U$ and $\mathbb{G}_{m,\mathbb{C}}^2$ are isomorphic and this fact is attributed to Serre. On the other hand, $U$ admits a surjective homomorphism of algebraic groups to $E$, so $U$ is not affine and hence cannot be isomorphic to $\mathbb{G}_{m,\mathbb{C}}^2$ as an algebraic group.

Corollary \ref{cor:main} is also false in positive characteristic as the well-known example $G = H = \mathbb{A}^2_{\mathbb{F}_2}$ shows, with both neutral elements equal to $(0,0)$ and the group laws given by the polynomials $x_1+x_2, y_1+y_2$ and $x_1+x_2, y_1+y_2+x_1x_2$ in the coordinates $x_1,y_1,x_2,y_2$ on $\mathbb{A}^4_{\mathbb{F}_2}$ respectively; there cannot exist an isomorphism of algebraic groups $G \to H$ since the kernels of multiplication by $2$ do not have the same dimension. More generally, if $p$ is an arbitrary prime, the additive group of Witt vectors of length $2$ over $\mathbb{F}_p$ is isomorphic to $\mathbb{G}_{a,\mathbb{F}_p}^2$ as a variety, but not as an algebraic group (one is $p$-torsion, the other is not).

The construction in the proof of Theorem \ref{thm:main} very much depends on the choice of base point and we now show that it does not extend to torsors. Recall that a torsor $T$ under a connected commutative algebraic group $G$ over a field $K$ of characteristic $0$ is a non-empty $K$-variety, i.e., a reduced and separated $K$-scheme of finite type, with an action $\alpha: T \times_K G \to T$ of $G$ such that the morphism $(\mathrm{pr}_{T},\alpha): T \times_K G \to T \times_K T$ is an isomorphism, where $\mathrm{pr}_{T}: T \times_K G \to T$ denotes the canonical projection to $T$.

\begin{prop}
Let $K$ be a field of characteristic $0$ and let $\mathfrak{C}'$ and $\mathfrak{D}'$ denote the categories defined as follows: they both have as objects pairs $(G,T)$ of a torsor $T$ under a connected commutative algebraic group $G$ over $K$. Given two such pairs $(G,T)$ and $(G',T')$, the morphisms from $(G,T)$ to $(G',T')$ in $\mathfrak{C}'$ are all pairs $(\phi, \psi)$ of a homomorphism of algebraic groups $\phi: G \to G'$ and a $\phi$-equivariant morphism $\psi: T \to T'$ while the morphisms from $(G,T)$ to $(G',T')$ in $\mathfrak{D}'$ are all morphisms of algebraic varieties from $T$ to $T'$. Let $\mathfrak{C}$, $\mathfrak{D}$, and $\mathcal{F}$ be as defined before Theorem \ref{thm:main}.

Then $\mathfrak{C}$ can be canonically identified with a subcategory of $\mathfrak{C}'$ and $\mathfrak{D}$ can be canonically identified with a subcategory of $\mathfrak{D}'$, the functor $\mathcal{F}$ canonically extends to a functor $\mathcal{F}': \mathfrak{C}' \to \mathfrak{D}'$, but the functor $\mathcal{G}$ from Theorem \ref{thm:main} cannot be extended to a functor $\mathcal{G}': \mathfrak{D}' \to \mathfrak{C}'$ that satisfies $\mathcal{G}' \circ \mathcal{F}' = \id_{\mathfrak{C}'}$.
\end{prop}

\begin{proof}
The canonical identifications of $\mathfrak{C}$ and $\mathfrak{D}$ with subcategories of $\mathfrak{C}'$ and $\mathfrak{D}'$ respectively are obtained by identifying $G$ with $(G,G)$, where $G$ is regarded as a trivial torsor under itself, and, in the case of $\mathfrak{C}$ and $\mathfrak{C}'$, identifying $\phi: G \to H$ with $(\phi,\phi)$. The functor $\mathcal{F}'$ fixes all objects and sends a morphism from $(G,T)$ to $(G',T')$ in $\mathfrak{C}'$ to the associated morphism from $T$ to $T'$.
	
Suppose that $\mathcal{G}': \mathfrak{D}' \to \mathfrak{C}'$ is a functor that extends $\mathcal{G}$ and satisfies $\mathcal{G}' \circ \mathcal{F}' = \id_{\mathfrak{C}'}$. For $a,b,c \in K$, let $\psi_a$ denote the morphism from $\mathbb{G}_{a,K} = \mathrm{Spec} K[X]$, considered as a trival torsor under itself, to itself that is induced by the polynomial $(X+a)^2-a^2$, let $\tau_b: \mathbb{G}_{a,K} \to \mathbb{G}_{a,K}$ denote translation by $b$, and let $m_c: \mathbb{G}_{a,K} \to \mathbb{G}_{a,K}$ denote the multiplication-by-$c$ homomorphism. We have $\psi_a = \tau_{-a^2} \circ \psi_0 \circ \tau_{a}$ and all translations are of course $\id_{\mathbb{G}_{a,K}}$-equivariant.

It follows that
\begin{equation}\label{eq:functoriality}
\mathcal{G}'(\psi_a) = (\id_{\mathbb{G}_{a,K}},\tau_{-a^2}) \circ \mathcal{G}'(\psi_0) \circ (\id_{\mathbb{G}_{a,K}},\tau_a)
\end{equation}
for all $a \in K$. At the same time, we have that $\mathcal{G}'(\psi_a) = (m_{2a},m_{2a})$ for all $a \in K$ since $\mathcal{G}'$ extends $\mathcal{G}$ (recall that, by construction, $\mathcal{G}$ sends a pointed morphism between powers of $\mathbb{G}_{a,K}$ to the homomorphism induced by its differential at the neutral element via the canonical isomorphisms between powers of $\mathbb{G}_{a,K}$ and their Lie algebras). Setting $a = 0,1$ in this equation and plugging the results into \eqref{eq:functoriality}, we deduce that $m_2 = m_0$, a contradiction. So no such $\mathcal{G}'$ exists.
\end{proof}

It would be interesting to know whether Corollary \ref{cor:main} could nevertheless be extended to torsors:
\end{rmk}

\begin{qstn}\label{qstn}
Let $K$ be a field of characteristic $0$. For $i = 1,2$, let $T_i$ be a torsor under a connected commutative algebraic group $G_i$ over $K$ and let $\alpha_i: T_i \times_K G_i \to T_i$ denote the corresponding $G_i$-action. Suppose that $T_1$ and $T_2$ are isomorphic as algebraic varieties. Must they then be isomorphic as torsors, i.e., must there exist an isomorphism of algebraic groups $\phi: G_1 \to G_2$ and an isomorphism of algebraic varieties $\psi: T_1 \to T_2$ such that $\alpha_2 \circ (\psi \times_K \phi) = \psi \circ \alpha_1$?
\end{qstn}

The following theorem and its proof were pointed out to us by B. Kahn as well as, in the semiabelian case, by one of the referees.

\begin{thm}\label{thm:partialanswer}
Question \ref{qstn} has a positive answer if
\begin{enumerate}
\item for some $i \in \{1,2\}$, every variety automorphism of the base change of $G_i$ to an algebraic closure of $K$ is a composition of an automorphism of algebraic groups with a translation, or
\item $G_1$ or $G_2$ is a vector group.
\end{enumerate}
\end{thm}

Theorem \ref{thm:partialanswer} applies if $G_1$ or $G_2$ is a semiabelian variety since, over an algebraic closure of $K$, every morphism from a connected algebraic group to a semiabelian variety is a composition of a homomorphism of algebraic groups with a translation, a fact which we will later generalize in certain situations. Later, we will also characterize all connected commutative algebraic groups over $K$ which satisfy condition (1) in Theorem \ref{thm:partialanswer}.

\begin{proof}
Suppose that $\psi: T_1 \to T_2$ is an isomorphism of algebraic varieties. Let $0_{G_1}$ and $0_{G_2}$ denote the respective neutral elements of $G_1$ and $G_2$. Fix an algebraic closure $\bar{K}$ of $K$ and base points $t_1 \in T_1(\bar{K})$ and $t_2 := \psi(t_1) \in T_2(\bar{K})$. By abuse of notation, we will often denote $K$-morphisms and their base changes to $\bar{K}$ by the same symbol. By the definition of a torsor, the restriction of $(\mathrm{pr}_{T_i},\alpha_i)_{\bar{K}}: (T_i)_{\bar{K}} \times_{\bar{K}} (G_i)_{\bar{K}} \to (T_i)_{\bar{K}} \times_{\bar{K}} (T_i)_{\bar{K}}$ to $\{t_i\} \times_{\bar{K}} (G_i)_{\bar{K}}$ induces an isomorphism $\zeta_i: (G_i)_{\bar{K}} \to (T_i)_{\bar{K}}$ ($i = 1,2$). We obtain an isomorphism of algebraic varieties $\chi := \zeta_2^{-1} \circ \psi \circ \zeta_1: (G_1)_{\bar{K}} \to (G_2)_{\bar{K}}$. By definition, we have
\[ \psi(\alpha_1(t_1,g)) = (\psi \circ \zeta_1)(g) = (\zeta_2 \circ \chi)(g) = \alpha_2(t_2,\chi(g))\]
for all $g \in G_1(\bar{K})$. We also know that $\chi(0_{G_1}) = 0_{G_2}$ since $\psi(t_1) = t_2$. Note that $(G_i)_{\bar{K}}$, and hence also $(T_i)_{\bar{K}}$, is irreducible by Summary 1.36 in \cite{MilneAG} since $G_i$ is connected ($i=1,2$).

Suppose first that $G_1$ or $G_2$ is a vector group. Since $(G_1)_{\bar{K}}$ and $(G_2)_{\bar{K}}$ are isomorphic as algebraic varieties, it follows from Proposition 14.51 in \cite{GoertzWedhorn} that both $G_1$ and $G_2$ are affine. Furthermore, by Corollary \ref{cor:main}, $(G_1)_{\bar{K}}$ and $(G_2)_{\bar{K}}$ are also isomorphic as algebraic groups. We deduce from this together with Corollary 14.9 in \cite{MilneAG} that both $G_1$ and $G_2$ are unipotent and hence, by Proposition 14.32 in \cite{MilneAG}, vector groups. By Propositions 5 and 6 in Sections 1 and 2 respectively of Chapter III of \cite{SerreGC}, the torsors $T_1$ and $T_2$ are both trivial, i.e., both have a $K$-rational point, and the theorem follows from Corollary \ref{cor:main}.

Suppose now that every variety automorphism of $(G_1)_{\bar{K}}$ is a composition of an automorphism of algebraic groups with a translation; note that this is equivalent to the same condition holding for the base change of $G_1$ to an arbitrary algebraic closure of $K$. By Corollary \ref{cor:main}, there exists an isomorphism of algebraic groups $\omega: (G_2)_{\bar{K}} \to (G_1)_{\bar{K}}$. It follows that $\omega \circ \chi$ is a composition of an automorphism of algebraic groups with a translation. Since $(\omega \circ \chi)(0_{G_1}) = 0_{G_1}$, $\omega \circ \chi$ is an automorphism of algebraic groups and therefore $\chi = \omega^{-1} \circ (\omega \circ \chi)$ is an isomorphism of algebraic groups. If every variety automorphism of $(G_2)_{\bar{K}}$ is a composition of an automorphism of algebraic groups with a translation, then it similarly follows that $\chi^{-1}$ is an isomorphism of algebraic groups. In any case, we obtain that $\chi$ is an isomorphism of algebraic groups.

The Galois group $\Gal(\bar{K}/K)$ acts on $\bar{K}$-varieties, $\bar{K}$-morphisms, and $\bar{K}$-points by base change in the usual way. For $K$-varieties $V$ and $W$, a morphism of $K$-varieties $f: V \to W$, and $\sigma \in \Gal(\bar{K}/K)$, we will identify $V_{\bar{K}}$, $W_{\bar{K}}$, and $f_{\bar{K}}$ with $\sigma\left(V_{\bar{K}}\right)$, $\sigma\left(W_{\bar{K}}\right)$, and $\sigma\left(f_{\bar{K}}\right)$ respectively via the canonical isomorphisms. Let now $\sigma \in \Gal(\bar{K}/K)$ and $g \in G_1(\bar{K})$ be arbitrary, but fixed. Set $h_i = \zeta_i^{-1}(\sigma(t_i))$ ($i = 1,2$). We have $\chi(h_1) = \zeta_2^{-1}(\psi(\sigma(t_1))) = \zeta_2^{-1}(\sigma(t_2)) = h_2$ since $\psi$ is defined over $K$ and $t_2 = \psi(t_1)$. Furthermore, it follows from the definition of $h_2$ that
\[ \sigma(\zeta_2)(\sigma(\chi)(g)) =  \alpha_2(\sigma(t_2),\sigma(\chi)(g)) = 
\alpha_2(\alpha_2(t_2,h_2),\sigma(\chi)(g)).\]
Since $\alpha_2$ is a group action, we deduce that
\[\sigma(\zeta_2)(\sigma(\chi)(g)) = \alpha_2(t_2,h_2+\sigma(\chi)(g)) = \zeta_2(h_2+\sigma(\chi)(g)).\]
On the other hand
\begin{align*}
\sigma(\zeta_2)(\sigma(\chi)(g)) =  \sigma(\zeta_2 \circ \chi)(g) = \sigma(\psi \circ \zeta_1)(g) = \psi(\sigma(\zeta_1)(g)) = \psi(\alpha_1(\sigma(t_1),g)) \\
= \psi(\alpha_1(\alpha_1(t_1,h_1),g)) = \psi(\alpha_1(t_1,h_1+g)) = (\psi \circ \zeta_1)(h_1+g) = (\zeta_2 \circ \chi)(h_1+g) \\
= \zeta_2(\chi(h_1)+\chi(g))
= \zeta_2(h_2+\chi(g)).
\end{align*}
Since $\zeta_2$ is an isomorphism, we obtain that $\sigma(\chi)(g) = \chi(g)$. This holds for all $g \in G_1(\bar{K})$ and so, by Theorem 3.37 in \cite{GoertzWedhorn}, $\chi = \sigma(\chi)$. So $\chi$ is $\Gal(\bar{K}/K)$-invariant and therefore descends to a morphism $\phi: G_1 \to G_2$ by Proposition 2.8 in \cite{Jahnel} or Theorem 14.70(1) and Section (14.20) in \cite{GoertzWedhorn} (note that $\chi$ automatically descends to a morphism defined over a finite Galois extension of $K$). It follows from Proposition 14.51 in \cite{GoertzWedhorn} that $\phi$ is an isomorphism of algebraic varieties and Theorem 14.70(1) in \cite{GoertzWedhorn} implies that $\phi$ is also an isomorphism of algebraic groups since $\chi$ is. We have
\begin{align*}
\alpha_2(\psi(t),\phi(g)) = \alpha_2(t_2,\zeta_2^{-1}(\psi(t))+\phi(g)) =
\zeta_2(\zeta_2^{-1}(\psi(t))+\phi(g)) \\ = (\psi \circ \zeta_1 \circ \chi^{-1})(\zeta_2^{-1}(\psi(t))+\phi(g))
=  \psi(\alpha_1(t_1,\chi^{-1}(\zeta_2^{-1}(\psi(t))+\phi(g)))) \\
= \psi(\alpha_1(t_1,\zeta_1^{-1}(t)+g)) = \psi(\alpha_1(\alpha_1(t_1,\zeta_1^{-1}(t)),g)) = \psi(\alpha_1(t,g))
\end{align*}
for all $t \in T_1(\bar{K})$ and $g \in G_1(\bar{K})$ and hence $\alpha_2 \circ (\psi \times_K \phi) = \psi \circ \alpha_1$ by Theorems 3.37 and 14.70(1) in \cite{GoertzWedhorn} as desired.
\end{proof}

As already mentioned, if the domain has no non-trivial unipotent group as a direct factor, we can give an explicit description of the sets of all (not necessarily pointed) morphisms and isomorphisms of algebraic varieties between two connected commutative algebraic groups over a field of characteristic $0$:

\begin{thm}\label{thm:explicit}
Let $G$ and $H$ be two connected commutative algebraic groups over a field $K$ of characteristic $0$ such that $G$ is not isomorphic to a product $U' \times_K G'$ of a non-trivial unipotent group $U'$ and an algebraic group $G'$.

Let $\phi: G \to H$ be a morphism of algebraic varieties. Then there exist a torus $T$, a commutative unipotent group $U$, homomorphisms of algebraic groups $p: G \to T$, $i: U \to H$, and $\psi: G \to H$, a translation $\tau: H \to H$, and a morphism of pointed algebraic varieties $\chi: T \to U$ (with the neutral element being the respective distinguished point) such that $\phi = \tau \circ \psi + i \circ \chi \circ p$.

Furthermore, $\tau$ and $\psi$ (and therefore $i \circ \chi \circ p$) are uniquely determined by these properties and by $\phi$. The unique functor $\mathcal{G}$ from Theorem \ref{thm:main} satisfies $\mathcal{G}(\tau^{-1} \circ \phi) = \psi$ and $\phi$ is an isomorphism of algebraic varieties if and only if $\psi$ is an isomorphism of algebraic groups.
\end{thm}

Over a base field of arbitrary characteristic, any morphism of algebraic varieties from a smooth connected algebraic group to a semiabelian variety is the composition of a homomorphism and a translation thanks to Lemma 5.4.8 and Remark 5.4.9(i) in \cite{Brion17}; see also Lemma 4.1 in \cite{Kahn}, Theorem 2 in \cite{Iitaka}, and Theorem 3 in \cite{Rosenlicht61}. Under the additional hypotheses that the base field is of characteristic $0$ and that the domain is commutative and has no non-trivial unipotent direct factor, Theorem \ref{thm:explicit} generalizes this fact in that it allows any connected commutative algebraic group as codomain, but yields the same conclusion if the codomain is a semiabelian variety.

If one could describe the set of isomorphisms of algebraic varieties between $G$ and $H$ explicitly without requiring $G$ to have no non-trivial unipotent direct factor, one would in particular obtain an explicit description of the automorphism group of affine space, which is a difficult open problem; the wide open Jacobian conjecture gives a conjectural characterization of this automorphism group (in characteristic $0$).

\begin{proof}[Proof of Theorem \ref{thm:explicit}]
Let $\phi: G \to H$ be an arbitrary morphism of algebraic varieties. Let $0_G$ and $0_H$ denote the respective neutral elements of $G$ and $H$.

By Theorem 3.4 in \cite{Brion09}, there exist a unipotent subgroup $U_1 \subset G$ and a torus $T_1 \subset G$ such that the addition homomorphism $U_1 \times_K T_1 \times_K G_{\ant} \to G$ is an isogeny. Since $U_1$ contains no non-trivial finite algebraic subgroups, the kernel of this isogeny must be contained in $\{0_G\} \times_K T_1 \times_K G_{\ant}$ so that
\[ G \simeq U_1 \times_K ((T_1 \times_K G_{\ant})/(T_1 \cap G_{\ant})).\]
Our hypothesis on $G$ now implies that $U_1$ is trivial and so the addition homomorphism $T_1 \times_K G_{\ant} \to G$ is an isogeny.

Fix an algebraic closure $\bar{K}$ of $K$. Again, by abuse of notation, we will often denote $K$-morphisms and their base changes to $\bar{K}$ by the same symbol. The morphism $\phi$ induces a morphism $\phi_1: T_1 \times_K G_{\ant} \to H$, sending $(t,g) \in T_1(\bar{K}) \times G_{\ant}(\bar{K})$ to $\phi(t+g)-\phi(g)$. For any fixed $g \in G_{\ant}(\bar{K})$, the induced morphism $(T_1)_{\bar{K}} \to (H/H_{\aff})_{\bar{K}}$ sends $0_G$ to the neutral element of $H/H_{\aff}$. By Corollary 3.6 in \cite{Milne}, this morphism is a homomorphism of algebraic groups since $(H/H_{\aff})_{\bar{K}}$ is an abelian variety. By Proposition 8.1 in \cite{MilneAG}, its image is an affine algebraic group. But since $(H/H_{\aff})_{\bar{K}}$ is an abelian variety, this image must be trivial by Example 8.4 in \cite{MilneAG}. Hence, $\phi_1$ has to factor through $H_{\aff}$.

But then for fixed $t \in T_1(\bar{K})$, the induced morphism $(G_{\ant})_{\bar{K}} \to (H_{\aff})_{\bar{K}}$ must be constant since $H_{\aff}$ is affine and $G_{\ant}$ is antiaffine (note that antiaffinity is preserved under field extensions thanks to Lemma 1.1 in \cite{Brion09}). It follows that $\phi(t+g) = \phi(g) + \phi_2(t)$ for some morphism $\phi_2: T_1 \to H_{\aff}$ and all $g \in G_{\ant}(\bar{K})$, $t \in T_1(\bar{K})$. By Corollary 16.15 in \cite{MilneAG}, there exist a unipotent group $U \subset H_{\aff}$ and a torus $T' \subset H_{\aff}$ such that the addition homomorphism $U \times_K T' \to H_{\aff}$ is an isomorphism. Using this isomorphism, we identify $U \times_K T'$ with $H_{\aff}$ so that $\phi_2 = (\phi_{2,1},\phi_{2,2})$ with $\phi_{2,1}: T_1 \to U$ and $\phi_{2,2}: T_1 \to T'$. Furthermore, $\phi_2(0_G) = 0_H$, which implies the same for $\phi_{2,1}$ and $\phi_{2,2}$.

The group $(T_1 \cap G_{\ant})(\bar{K})$ is finite. Suppose that $t \in (T_1 \cap G_{\ant})(\bar{K})$ and let $n \in \mathbb{N}$ denote the order of $t$. For any $t' \in T_1(\bar{K})$, we have
\[n(\phi_2(t'+t)-\phi_2(t')) = n(\phi(t')-\phi(-t)-\phi(t')+\phi(0_G)) = n(\phi(0_G)-\phi(-t))\]
by definition of $\phi_2$. Since $\phi|_{G_{\ant}}$ is the composition of a homomorphism and a translation by Lemma 1.5(i) in \cite{Brion09}, we also have that
\[ n(\phi(0_G)-\phi(-t)) = \phi(0_G) - \phi(n(-t)) = \phi(0_G) - \phi(0_G) = 0_H.\]
Putting the two chains of equality together, we deduce that $n(\phi_2(t'+t)-\phi_2(t')) = 0_H$. Since $U(\bar{K})$ is torsion-free, $\phi_{2,1}(t'+t) = \phi_{2,1}(t')$. It follows that the morphism $(T_1 \cap G_{\ant})_{\bar{K}} \times_{\bar{K}} (T_1)_{\bar{K}} \to U_{\bar{K}}$ that is given by the addition homomorphism of $(T_1)_{\bar{K}}$ composed with $(\phi_{2,1})_{\bar{K}}$ is the same as the morphism that is given by the projection to $(T_1)_{\bar{K}}$ followed by $(\phi_{2,1})_{\bar{K}}$ thanks to Theorem 3.37 in \cite{GoertzWedhorn}, applied to the restrictions of the two morphisms to each connected component of $(T_1 \cap G_{\ant})_{\bar{K}} \times_{\bar{K}} (T_1)_{\bar{K}}$. By Theorem 14.70(1) in \cite{GoertzWedhorn}, the two corresponding morphisms $(T_1 \cap G_{\ant}) \times_K T_1 \to U$ coincide as well. So, by Proposition 5.21 in \cite{MilneAG}, $\phi_{2,1}$ factors through the torus $T := T_1/(T_1 \cap G_{\ant}) = G/G_{\ant}$ and is induced by a morphism of algebraic varieties $\chi: T \to U$ which sends the neutral element of $T$ to $0_H$.

Let $p: G \to T$ denote the canonical quotient homomorphism, let $i: U \to H$ denote the canonical inclusion homomorphism, and let $\tau: H \to H$ denote the translation by $\phi(0_G)$. Set $\psi = \tau^{-1} \circ (\phi - i \circ \chi \circ p)$ so that $\phi = \tau \circ \psi + i \circ \chi \circ p$.

By construction, we have that
\begin{equation}\label{eq:someequation}
\psi(t+g) = \phi(g)-\phi(0_G)+\phi_{2,2}(t)
\end{equation}
for all $g \in G_{\ant}(\bar{K})$ and all $t \in T_1(\bar{K})$. Furthermore, the addition homomorphism $T_1 \times_K G_{\ant} \to G$ is surjective.

Thus, if $\tilde{g}_1, \tilde{g}_2 \in G(\bar{K})$, then there exist $t_1, t_2 \in T_1(\bar{K})$ and $g_1, g_2 \in G_{\ant}(\bar{K})$ such that $\tilde{g}_i = t_i + g_i$ for $i = 1,2$. Using  \eqref{eq:someequation}, we deduce that
\[ \psi(\tilde{g}_1+\tilde{g}_2) = \psi((t_1+t_2)+(g_1+g_2)) = \phi(g_1+g_2)-\phi(0_G)+\phi_{2,2}(t_1+t_2).\]
Since $\phi_{2,2}(0_G) = 0_H$, it follows from Proposition 12.49 in \cite{MilneAG} that $\phi_{2,2}$ is a homomorphism of algebraic groups and hence
\[ \psi(\tilde{g}_1+\tilde{g}_2) = \phi(g_1+g_2)-\phi(0_G)+\phi_{2,2}(t_1)+\phi_{2,2}(t_2).\]
But Lemma 1.5(i) in \cite{Brion09} implies that $\phi(g_1+g_2) - \phi(0_G) = \phi(g_1)-\phi(0_G) + \phi(g_2)-\phi(0_G)$ and it follows that
\[ \psi(\tilde{g}_1+\tilde{g}_2) = (\phi(g_1)-\phi(0_G)+\phi_{2,2}(t_1))+(\phi(g_2)-\phi(0_G)+\phi_{2,2}(t_2)).\]
Applying \eqref{eq:someequation} again, we deduce that
\[ \psi(\tilde{g}_1+\tilde{g}_2) = \psi(t_1+g_1) + \psi(t_2+g_2) = \psi(\tilde{g}_1) + \psi(\tilde{g}_2).\]

It follows that $\psi$ is a homomorphism of algebraic groups. We compute that the functor $\mathcal{G}$ from Theorem \ref{thm:main} satisfies
\[ \mathcal{G}(\tau^{-1} \circ \phi) = \mathcal{G}(\psi + i \circ \chi \circ p) = \mathcal{G}(\psi) + i \circ \mathcal{G}(\chi) \circ p = \mathcal{G}(\psi) = \psi\]
as desired since the homomorphism of algebraic groups $\mathcal{G}(\chi): T \to U$ must be trivial thanks to Corollary 14.18(b) in \cite{MilneAG}. This implies that $\psi$ is an isomorphism of algebraic groups if $\phi$ is an isomorphism of algebraic varieties.

The uniqueness of $\tau$ is clear since $\tau: H \to H$ is a translation, so determined by $\tau(0_G)$, and $\tau(0_G) = \phi(0_G)$. The uniqueness of $\psi$ follows from the fact that a difference of two compositions of the form $i \circ \chi \circ p$ is a homomorphism of algebraic groups if and only if it is trivial thanks to Corollary 14.18(b) in \cite{MilneAG}.

Finally, suppose that $\psi$ is an isomorphism of algebraic groups. We want to show that $\phi$ is an isomorphism of algebraic varieties. We can assume without loss of generality that $\tau = \id_H$. It follows from Corollary 14.18(a) in \cite{MilneAG} that the homomorphism of algebraic groups $p \circ \psi^{-1} \circ i: U \to T$ is trivial. This implies that
\[ (\id_H - i \circ \chi \circ p \circ \psi^{-1}) \circ (\id_H + i \circ \chi \circ p \circ \psi^{-1}) = (\id_H + i \circ \chi \circ p \circ \psi^{-1}) \circ (\id_H - i \circ \chi \circ p \circ \psi^{-1}) = \id_H\]
as well as $(\id_H - i \circ \chi \circ p \circ \psi^{-1}) \circ \phi = \psi$. So $\phi$ is an isomorphism of algebraic varieties and we are done.
\end{proof}

We can now characterize the connected commutative algebraic groups that satisfy condition (1) in Theorem \ref{thm:partialanswer}:

\begin{thm}\label{thm:characterization}
Let $G$ be a connected commutative algebraic group over a field $K$ of characteristic $0$. The following are equivalent:
\begin{enumerate}
\item every variety automorphism of $G$ is a composition of an automorphism of algebraic groups with a translation,
\item every variety automorphism of the base change of $G$ to an algebraic closure of $K$ is a composition of an automorphism of algebraic groups with a translation, and
\item $G$ is antiaffine or $G$ is a semiabelian variety or there exists an antiaffine semi\-abelian variety $H$ over $K$ such that $G \simeq \mathbb{G}_{a,K} \times_K H$ as algebraic groups.
\end{enumerate}
\end{thm}

\begin{proof}
We first show that (1) implies (3). We take $n \in \mathbb{N} \cup \{0\}$ maximal such that $G \simeq \mathbb{G}_{a,K}^n \times_K H$ as algebraic groups for some algebraic group $H$ over $K$. Being a quotient of $G$, $H$ is automatically connected and commutative. If $n \geq 2$, then there exist automorphisms of $\mathbb{G}_{a,K}^n$ that fix the neutral element, but are not automorphisms of algebraic groups, e.g., $(X_1,\hdots,X_n) \mapsto (X_1+X_2^2,X_2,\hdots,X_n)$, a contradiction with (1). So $n \leq 1$.

Suppose first that $H$ is not a semiabelian variety. Hence, there exists an injective homomorphism of algebraic groups $i: \mathbb{G}_{a,K} \hookrightarrow H$. Any non-trivial $K$-torus, being affine, admits a non-constant morphism of algebraic varieties to $\mathbb{G}_{a,K}$. By post-composing with a translation, we can assume that this morphism sends the neutral element to the neutral element. Hence, if $H$ has a non-trivial torus quotient, then Theorem \ref{thm:explicit} implies that there exists an automorphism of the algebraic variety $H$ that fixes its neutral element, but is not an automorphism of algebraic groups, another contradiction with (1). So in particular, $H/H_{\ant}$ has no non-trivial torus quotient, so must be a vector group by Proposition \ref{prop:summary} and Proposition 14.32 and Theorem 16.13(b) in \cite{MilneAG}.

By the same Proposition 14.32 and Theorem 16.13(b) in \cite{MilneAG}, $H_{\aff}$ is isomorphic to a product of an algebraic group of multiplicative type and a vector group. Since there are no non-trivial homomorphisms of algebraic groups from an algebraic group of multiplicative type to a unipotent algebraic group by Corollary 14.18(b) in \cite{MilneAG} and since any algebraic subgroup of a vector group over $K$ is a direct factor of that vector group, the canonical homomorphism $H_{\aff} \to H/H_{\ant}$, which is surjective by Proposition 3.1(i) in \cite{Brion09}, admits a section. It follows that $H \simeq H_{\ant} \times_K (H/H_{\ant})$. Since $n$ was taken maximal, we deduce that $H = H_{\ant}$. If $n = 1$, then we can consider the morphism of algebraic varieties $\id_G + i \circ \phi \circ p$, where $p: G \to G/H \simeq \mathbb{G}_{a,K}$ is the canonical projection and $\phi: \mathbb{G}_{a,K} \to \mathbb{G}_{a,K}$ is an arbitrary morphism of algebraic varieties that fixes the neutral element, but is not a homomorphism of algebraic groups. Since $p \circ i$ is the zero morphism, we have that
\[ (\id_G + i \circ \phi \circ p) \circ (\id_G - i \circ \phi \circ p) = (\id_G - i \circ \phi \circ p) \circ (\id_G + i \circ \phi \circ p) = \id_G,\]
so $\id_G + i \circ \phi \circ p$ is an automorphism of algebraic varieties. Furthermore, $\id_G + i \circ \phi \circ p$ fixes the neutral element. But $\id_G + i \circ \phi \circ p$ is not an automorphism of algebraic groups, otherwise it would follow that $\phi$ is a homomorphism of algebraic groups. We obtain a contradiction with (1) and it follows that $n = 0$. So $G$ is antiaffine and we are done.

Suppose now that $H$ is a semiabelian variety. If $n = 0$, then (3) holds, so suppose that $n = 1$. If $H$ is antiaffine, then (3) holds, so let us suppose that $H$ is not antiaffine and aim to obtain a contradiction. It follows that $H/H_{\ant}$ is a non-trivial algebraic group. By Proposition \ref{prop:summary}, $H/H_{\ant}$ is affine. By Proposition 3.1(i) in \cite{Brion09}, the canonical homomorphism $H_{\aff} \to H/H_{\ant}$ is surjective. Since $H$ is a semiabelian variety, $H_{\aff}$ is a torus and so $H/H_{\ant}$ is a non-trivial torus. As above, there exists a non-constant morphism of algebraic varieties $\psi: H/H_{\ant} \to \mathbb{G}_{a,K}$ that sends the neutral element to the neutral element. By Corollary 14.18(b) in \cite{MilneAG}, $\psi$ is not a homomorphism of algebraic groups. Let $\pi: G \simeq \mathbb{G}_{a,K} \times_K H \to (\mathbb{G}_{a,K} \times_K H)/(\mathbb{G}_{a,K} \times_K H_{\ant}) \simeq H/H_{\ant}$ denote the canonical projection and let $\iota: \mathbb{G}_{a,K} \simeq \mathbb{G}_{a,K} \times_K \{0_H\} \hookrightarrow \mathbb{G}_{a,K} \times_K H \simeq G$ denote the canonical inclusion, where $0_H$ denotes the neutral element of $H$. Consider the morphism of algebraic varieties $\id_G + \iota\circ \psi \circ \pi: G\to G$. Since $\pi \circ\iota$ is the zero morphism, we have that
\[ (\id_G + \iota \circ \psi \circ \pi) \circ (\id_G - \iota \circ \psi \circ \pi) = (\id_G - \iota \circ \psi \circ \pi) \circ (\id_G + \iota \circ \psi \circ \pi) = \id_G,\]
so $\id_G + \iota \circ \psi \circ \pi$ is an automorphism of algebraic varieties. Furthermore, $\id_G + \iota \circ \psi \circ \pi$ fixes the neutral element. But $\id_G + \iota \circ \psi \circ \pi$ is not an automorphism of algebraic groups, otherwise it would follow that $\psi$ is a homomorphism of algebraic groups. We obtain a contradiction with (1) as desired.

We now show that (3) implies (1). If $G$ is antiaffine, (1) follows from Lemma 1.5(i) in \cite{Brion09}. If $G$ is a semiabelian variety, (1) follows from Lemma 5.4.8 and Remark 5.4.9(i) in \cite{Brion17}. So suppose that there exists an antiaffine semiabelian variety $H$ such that $G \simeq \mathbb{G}_{a,K} \times_K H$. We will identify $G$ with $\mathbb{G}_{a,K} \times_K H$ for simplicity. Let $\phi: G \to G$ be any automorphism of algebraic varieties. We can assume that $\phi$ fixes the neutral element and want to show that $\phi$ is an automorphism of algebraic groups. Fix an algebraic closure $\bar{K}$ of $K$. By Remark 5.4.2(iii) and Proposition 5.4.5 in \cite{Brion17}, every morphism of algebraic varieties from $\mathbb{G}_{a,\bar{K}}$ to $H_{\bar{K}}$ is constant. Since $H$ is antiaffine and therefore $H_{\bar{K}}$ is antiaffine by Lemma 1.1 in \cite{Brion09}, every morphism from $H_{\bar{K}}$ to $\mathbb{G}_{a,\bar{K}}$ is constant. It follows that $\phi = \phi_1 \times_K \phi_2$, where $\phi_1: \mathbb{G}_{a,K} \to \mathbb{G}_{a,K}$ and $\phi_2: H \to H$ are automorphisms of algebraic varieties that fix the respective neutral element. By what we have already shown, $\phi_2$ is an automorphism of algebraic groups. On the other hand, $\phi_1$ is an automorphism of algebraic varieties that fixes the neutral element, so must be induced by a linear polynomial with vanishing constant term and must therefore also be an automorphism of algebraic groups. It follows that $\phi$ is an automorphism of algebraic groups and so (1) holds.

We now show that (2) and (3) are equivalent by showing that condition (3) is equivalent to the same condition for the base change of $G$ to some algebraic closure of $K$. By Lemma 1.1 in \cite{Brion09}, an algebraic group over $K$ is antiaffine if and only if its base change to some algebraic closure of $K$ is antiaffine. By Remark 5.4.2(iii) and Lemma 5.4.3 in \cite{Brion17}, an algebraic group over $K$ is a semiabelian variety if and only if its base change to some algebraic closure of $K$ is a semiabelian variety. Finally, if there exists an antiaffine semiabelian variety $H$ over $K$ such that $G \simeq \mathbb{G}_{a,K} \times_K H$ as algebraic groups, then $G_{\ant} \simeq H$ is a semiabelian variety and $G/G_{\ant} \simeq \mathbb{G}_{a,K}$ as algebraic groups. On the other hand, if $G_{\ant}$ is a semiabelian variety and $G/G_{\ant} \simeq \mathbb{G}_{a,K}$ as algebraic groups, then it follows by an analogous argument as the one for $H \simeq H_{\ant} \times_K (H/H_{\ant})$ earlier in this proof that $G \simeq (G/G_{\ant}) \times_K G_{\ant} \simeq \mathbb{G}_{a,K} \times_K G_{\ant}$ as algebraic groups. Since affineness is preserved under field extension, it follows from Lemma 1.1 in \cite{Brion09} and Proposition \ref{prop:summary} that the formation of $G_{\ant}$ commutes with field extension. Corollary 14.9 and Proposition 14.32 in \cite{MilneAG} imply that a connected affine commutative algebraic group over $K$ is isomorphic to $\mathbb{G}_{a,K}$ as an algebraic group if and only if its base change to some algebraic closure $\bar{K}$ of $K$ is isomorphic to $\mathbb{G}_{a,\bar{K}}$. All together, this shows that condition (3) is equivalent to the analogous condition for the base change of $G$ to some algebraic closure of $K$. By what we have shown already, this implies that (2) and (3) are equivalent.
\end{proof}

\section*{Acknowledgements}
I thank Michel Brion, Ziyang Gao, Marc Hindry, Bruno Kahn, Hanspeter Kraft, and Zev Rosengarten for helpful conversations and correspondence. I especially thank Bruno Kahn for suggesting to generalize Corollary \ref{cor:main} to Theorem \ref{thm:main} and to consider torsors. I thank the referees for their useful comments and suggestions, in particular for pointing out to me the proof of the automatic additivity and uniqueness of $\mathcal{G}$ in Theorem \ref{thm:main}, and I thank both Bruno Kahn and, in the semiabelian case, one of the referees for suggesting Theorem \ref{thm:partialanswer} and its proof. This work was supported by the Swiss National Science Foundation through the Early Postdoc.Mobility grant no. P2BSP2\_195703. I thank the Mathematical Institute of the University of Oxford and my host there, Jonathan Pila, for hosting me as a visitor for the duration of this grant. During the final revisions of this manuscript, I was employed by the Leibniz Universit\"at Hannover on the ERC grant of Ziyang Gao.

\vspace{\baselineskip}
\noindent
\framebox[\textwidth]{
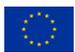
\begin{tabular*}{0.96\textwidth}{@{\extracolsep{\fill} }cp{0.84\textwidth}}
\raisebox{-0.7\height}{%
    \begin{tikzpicture}[y=0.80pt, x=0.8pt, yscale=-1, inner sep=0pt, outer sep=0pt, 
    scale=0.12]
    \definecolor{c003399}{RGB}{0,51,153}
    \definecolor{cffcc00}{RGB}{255,204,0}
    \begin{scope}[shift={(0,-872.36218)}]
      \path[shift={(0,872.36218)},fill=c003399,nonzero rule] (0.0000,0.0000) rectangle (270.0000,180.0000);
      \foreach \myshift in 
           {(0,812.36218), (0,932.36218), 
    		(60.0,872.36218), (-60.0,872.36218), 
    		(30.0,820.36218), (-30.0,820.36218),
    		(30.0,924.36218), (-30.0,924.36218),
    		(-52.0,842.36218), (52.0,842.36218), 
    		(52.0,902.36218), (-52.0,902.36218)}
        \path[shift=\myshift,fill=cffcc00,nonzero rule] (135.0000,80.0000) -- (137.2453,86.9096) -- (144.5106,86.9098) -- (138.6330,91.1804) -- (140.8778,98.0902) -- (135.0000,93.8200) -- (129.1222,98.0902) -- (131.3670,91.1804) -- (125.4894,86.9098) -- (132.7547,86.9096) -- cycle;
    \end{scope}
    \end{tikzpicture}%
}
&
Gabriel Dill has received funding from the European Research Council (ERC) under the European Union's Horizon 2020 research and innovation programme (grant agreement n$^\circ$ 945714).
\end{tabular*}

}

\section*{Conflicts of Interest}
The author declares that he has no conflict of interest.

\bibliographystyle{amsalpha}

\bibliography{Bibliography}

\end{document}